\documentclass[11pt]{article}
\usepackage{amsmath,amssymb,amsfonts,amsthm}
\usepackage{float}
\numberwithin{equation}{section}
\newtheorem{theorem}{Theorem}[section]
\newtheorem{definition}{Definition}[section]
\newtheorem{lemma}[theorem]{Lemma}
\newtheorem{proposition}[theorem]{Proposition}
\newtheorem{corollary}[theorem]{Corollary}

\begin{document}
\begin{center}
{\Large{\textbf{{Characteristics of Finite Jaco Graphs, $J_n(1), n\in \Bbb N$}}}} 
\end{center}
\vspace{0.5cm}
\large{\centerline{(Johan Kok, Paul Fisher, Bettina Wilkens, Mokhwetha Mabula, Vivian Mukungunugwa)\footnote {\textbf {Affiliation of authors:}\\
\noindent Johan Kok (Tshwane Metropolitan Police Department), City of Tshwane, Republic of South Africa\\
e-mail: kokkiek2@tshwane.gov.za\\ \\
\noindent Paul Fisher (Department of Mathematics, University of Botswana), City of Gaborone, Republic of Botswana\\
e-mail: paul.fisher@mopipi.ub.bw\\ \\
\noindent Bettina Wilkens (Department of Mathematics, University of Botswana), City of Gaborone, Republic of Botswana \\
e-mail: wilkensb@mopipi.ub.bw\\ \\
\noindent Mokhwetha Mabula (Department of Mathematics and Applied Mathematics, University of Pretoria), City of Tshwane, Republic of South Africa\\ 
e-mail: mokhwetha.Mabula@up.ac.za\\ \\
\noindent Vivian Mukungunugwa (Department of Mathematics, University of Zimbabwe), City of Harare, Republic of Zimbabwe\\
e-mail: vivianm@maths.uz.ac.zw}}
\vspace{0.5cm}
\begin{abstract}
\noindent We introduce the concept of a family of finite directed graphs (\emph{order 1}) which are directed graphs derived from a infinite directed graph (\emph{order 1}), called the \emph{1}-root digraph. The \emph{1}-root digraph has four fundamental properties which are; $V(J_\infty(1)) = \{v_i|i \in \Bbb N\}$ and, if $v_j$ is the head of an edge (arc) then the tail is always a vertex $v_i, i<j$ and, if $v_k,$ for smallest $k \in \Bbb N$ is a tail vertex then all vertices $v_ \ell, k< \ell<j$ are tails of arcs to $v_j$ and finally, the degree of vertex $k$ is $d(v_k) = k.$ The family of finite directed graphs are those limited to $n \in \Bbb N$ vertices by lobbing off all vertices (and edges arcing to vertices) $v_t, t > n.$ Hence, trivially we have $d(v_i) \leq i$ for $i \in \Bbb N.$ We present an interesting Fibonaccian-Zeckendorf result and present the Fisher Algorithm to table particular values of interest. It is meant to be an \emph {introductory paper} to encourage exploratory research.
\end{abstract}
\noindent {\footnotesize \textbf {Keywords:} Jaco graph, Directed graph, Jaconian vertex, Jaconian set, Number of edges, Shortest path, Fisher Algorithm, Zeckendorf representation}\\ \\
\noindent {\footnotesize \textbf {AMS Classification Numbers:} 05C07, 05C12, 05C20, 11B39} 
\section{Introduction}
\noindent We introduce the concept of a family of finite Jaco Graphs (\emph{order 1}) which are directed graphs derived from the infinite Jaco Graph (\emph{order 1}), called the \emph{1}-root digraph. The \emph{1}-root digraph has four fundamental properties which are; $V(J_\infty(1)) = \{v_i|i \in \Bbb N\}$ and, if $v_j$ is the head of an edge (arc) then the tail is always a vertex $v_i, i<j$ and, if $v_k,$ for smallest $k \in \Bbb N$ is a tail vertex then all vertices $v_ \ell, k< \ell<j$ are tails of arcs to $v_j$ and finally, the degree of vertex $k$ is $d(v_k) = k.$
\begin{definition}
The infinite Jaco Graph $J_\infty(1)$ is defined by $V(J_\infty(1)) = \{v_i| i \in \Bbb N\}$, $E(J_\infty(1)) \subseteq \{(v_i, v_j)| i, j \in \Bbb N, i< j\}$ and $(v_i,v_ j) \in E(J_\infty(1))$ if and only if $2i - d^-(v_i) \geq j.$
\end{definition}
\begin{definition}
The family of finite Jaco Graphs are defined by $\{J_n(1) \subseteq J_\infty(1)|n\in \Bbb {N}\}.$ A member of the family is referred to as the Jaco Graph, $J_n(1).$
\end{definition} 
\begin{definition}
The set of vertices attaining degree $\Delta (J_n(1))$ is called the Jaconian vertices of the Jaco Graph $J_n(1),$ and denoted, $\Bbb{J}(J_n(1))$ or, $\Bbb{J}_n(1)$ for brevity.
\end{definition}
\begin{definition}
The lowest numbered (indiced) Jaconian vertex is called the prime Jaconian vertex of a Jaco Graph.
\end{definition}
\begin{definition}
If $v_i$ is the prime Jaconian vertex of a Jaco Graph $J_n(1)$, the complete subgraph on vertices $v_{i+1}, v_{i+2}, \cdots,v_n$ is called the Hope subgraph of a Jaco Graph and denoted,  $\Bbb{H}(J_n(1))$ or, $\Bbb{H}_n(1)$ for brevity.
\end{definition}
\begin{definition} 
If, in applying definition 1.1 to vertex $v_i$  (not necessarily exhaustively), or for logical method of proof we have the edge $(v_i, v_k)$ linked in a Jaco Graph $J_n(1),$ then the degree vertex $v_i$ attains at $v_k$ is called the, "$\emph{at degree of $v_i$}$ at $v_k$", and is denoted, $d^*(v_i)@v_k.$
\end{definition}
\begin{definition}
In $J_\infty(1)$ we have $n = d^+(v_n) + d^-(v_n)$ whilst in $J_n(1)$ we have $d(v_i) = \lceil d^+(v_i)\rceil + d^-(v_i), i\leq n.$ 
\end{definition}
\noindent {\bf Property 1:} From the definition of a Jaco Graph $J_n(1),$ it follows that for the prime Jaconian vertex $v_i,$ we have $ d(v_m)= m$ for all $m\in\{1,2,3,\cdots,i\}.$\\ \\
{\bf Property 2:} From the definition of a Jaco Graph $J_n(1),$ it follows that $\Delta (J_k(1))\leq \Delta (J_n(1))$ for all $k\leq n.$\\ \\
{\bf Property 3:} The $d^-(v_k)$ for any vertex $v_k$ of a Jaco Graph $J_n(1),~ n\geq k$ is equal to $d(v_k)$ in the underlying Jaco Graph $J_k(1).$\\ \\

\begin{lemma}
If in a Jaco Graph $J_n(1),$ and for smallest $i$ with $d(v_i) = i,$ the edge $(v_i, v_n)$ is defined, then $v_i$ is the prime Jaconian vertex of $J_n(1).$
\end{lemma}
\begin{proof}
If by definition 1.1 and for smallest $i$ with $d(v_i) = i$ the edge $(v_i, v_n)$ is defined, we have in the underlying graph of $J_n(1)$ that $d(v_j) \leq d(v_i)$ for all $j > i$. We also have that $d(v_s) < d(v_i), s< i.$ So it follows that $d(v_i) = \Delta(J_n(1))$ hence by definition 1.4 the vertex $v_i$ is the prime Jaconian vertex of $J_n(1).$ 
\end{proof}
\begin{lemma}
For all Jaco Graphs $J_n(1),~n\geq2$ and, $v_i, v_{i-1}\in V(J_n(1))$ we have that in the underlying graph $|(d(v_i) - d(v_{i-1})|\leq 1.$ 
\end{lemma}
\begin{proof}
Consider the Jaco Graph $J_n(1),~n\geq2.$ The result is trivially true for all vertices $v_1,v_2,v_3,\cdots,v_k$ if $v_k$ is the prime Jaconian vertex of $J_n(1).$ Now consider the Hope subraph $\Bbb{H}(J_n(1)).$ All vertices of $\Bbb{H}(J_n(1))$ have equal degree so the result holds for the Hope subraph \emph{per se}. Furthermore if a vertex $v_j,~(k+1)\leq j\leq n$ is linked to a vertex $v_t,~1\leq t\leq k$ then all vertices $v_l,~(k+1)\leq l< j$ are linked to $v_t$ which implies 
$|d(v_j)-d(v_l)|=0 <1$ hence $|(d(v_{j+1})-d(v_j)|\leq1.$
\end{proof}
\begin{corollary}
For a Jaco Graph $J_n(1)$ the maximum degree $\Delta (J_n(1))$ might repeat itself  as $n$ increases to $n+1,$ (i.e. $\Delta J_n(1) = \Delta J_{n+1}(1)$) but on an increase of we always obtain $\Delta J_{n+1}(1) = \Delta (J_n(1))+1.$ 
\end{corollary}
\begin{proof}
The result follows from Lemma 1.2.
\end{proof}
\section{The Fisher Algorithm for $\{J_i(1), i\in\{4,5,6,\cdots, s\in \Bbb N\}$}
\noindent The family of finite Jaco Graphs are those limited to $n \in \Bbb N$ vertices by lobbing off all vertices (and edges arcing to vertices) $v_t, t > n.$ Hence, trivially we have $d(v_i) \leq i$ for $i \in \Bbb N.$ \\ \\
Column 1 is the map: $\phi(v_i) \rightarrow i, \forall i.$  \\
Column 2 is the in-degree of vertex $v_i.$ \\
Column 3 is the out-degree of vertex $v_i$ in $J_\infty(1).$ \\
Column 4 is the set $\Bbb J(J_i(1)).$\\
Column 5 is $\Delta(J_i(1)).$ \\
Column 6 is the distance $d_{J_i(1)}(v_1,v_i).$\\ \\
\noindent We generally refer to the entries in a row $i$ as: $ent_{1i} =i,$ $ent_{2i} = d^-(v_i),~ent_{3i} = d^+(v_i),~ent_{4i} = \Bbb J(J_i(1)),~ ent_{5i} =\Delta(J_i(1)) ,~ent_{6i} = d_{J_i(1)}(v_1,v_i)$ as interchangeable.\\ \\
Note that rows 1, 2 and 3 follow easily from definition 1.1. \\ \\
Step 0: Set $j = 4$, then set $i = j$ and $s\geq 4.$ \\
Step 1: Set $ent_{1i} = i.$ \\
Step 2: Set $ent_{2i}=ent_{1(i-1)} - ent_{5(i-1)}.$ (Note that $d^-(v_i)=v(\Bbb{H}_{i-1}(1))=(i-1)- \Delta(J_{i-1}(1))).$\\
Step 3: Set $ent_{3i}=ent_{1i} - ent_{2i}.$ (Note that $d^+(v_i)=i-d^-(v_i)).$\\
Step 4: Consider $ent_{4(i-1)}.$  If $ent_{4(i-1)}=\{v_k\},$ set $t=k,$ else set $t=k+1.$\\  
Step 5: Set the prime Jaconian vertex as $v_t$ so $\Bbb{J}(J_i(1))=\{v_t\}$ to begin with. Let $l=t+1,~t+ 2,\cdots,i-1$ and  recursively calculate $i - ent_{1l} + ent_{2l}$. If $i-ent_{1l}+ent_{2l}=t,$ add $v_l$ to the set of Jaconian vertices, else go to Step 6. \\
Step 6: Set $ent_{5i}= t.$ (Note that if $\Bbb J_i(1) =\{v_t, v_{t+1}, .. v_\ell\},$ then, $\Delta(J_i(1))= t).$ \\
Step 7: Select smallest $k$ such that, $k+ ent_{3k}\geq i$ then set $ent_{6i}=ent_{6k}+1.$ \\
Step 8: Set $j = i+1$, then set $i = j.$  If $i \leq s$, go to Step 1, else go to Step 9.\\
Step 9: Exit. \\ \\
\begin{proposition}
Consider the Jaco Graph $J_i(1),~i\geq 4.$ If the Jaconian vertex of $J_{i-1}(1)$ is unique say, $v_k$ then $k + d^+(v_k)<i$ and $(k+1)+ d^+(v_{k+1}) > i.$
\end{proposition}
\begin{proof}
Because the Jaconian vertex $v_k$ is unique to $J_{i-1}(1)$ it implies that edge $(v_k,v_{i-1})$ exists (see Theorem 2.11), so $k+d^+(v_k)= i-1<i.$  And since edge $(v_{k+1},v_i)$ does not exist in $J_{i-1}(1)$ we have $d(v_{k+1})= k-1.$ \\ \\   
By extending to $J_i(1)$ the edge $(v_{k+1},v_i)$ is linked. So degree of $v_{k+1}$ increases to $d(v_{k+1})=k$ implying $d^+(v_{k+1})$ increased by $1.$ Thus, $(k+1) + d^+(v_{k+1})= (k+1)+(d^+(v_k)+1)=(i-1)+2=i+1>i.$  
\end{proof}
\begin{lemma}
(Conjectured): If for $n \in \Bbb N$ we have that $d^+(v_n) = \ell$ is non-repetitive (meaning $d^+(v_{n-1}) < d^+(v_n) < d^+(v_{n+1})$) then, $\Bbb J(J_n(1)) = \{v_\ell \}.$
\end{lemma}
\begin{theorem}
(Morrie's Theorem): If a Jaco Graph $J_n(1), n \geq 2$ has a prime Jaconian vertex $v_k$ then:\\ \\
(a) $d^-(v_k)=d^-(v_{k+1})$ and $d^-(v_{k+2})=d^-(v_{k+1})+1$ if and only if $\Bbb J(J_n) = \{v_k\}$ and $\Bbb J(J_{n+1}(1)) = \{v_k, v_{k+1},v_{k+2}\}$, \\ 
(b) $d^-(v_k)= d^-(v_{k+1}) = d^-(v_{k+2})$ if and only if $\Bbb J(J_n) = \{v_k\}$ and $\Bbb J(J_{n+1}(1)) = \{v_k, v_{k+1}\}.$
\end{theorem}
\begin{proof}
Let $d^-(v_k)=d^-(v_{k+1})$ and $d^-(v_{k+2})=d^-(v_{k+1})+1$ for the Jaco Graph $J_n(1), n\geq 2.$ and let $\Bbb J(J_n(1)) = \{v_k\}.$ From definition 1.7 and Steps 1, 2 and 3 of the Fisher Algorithm it follow that we have associated entries:\\ \\
\noindent $ent_{1k} = k, ent_{2k} = d^-(v_k), ent_{3k} = d^+(v_k) = k - d^-(v_k)$ and,\\
$ent_{1(k+1)} = k +1, ent_{2(k+1)} = d^-(v_{k+1})= d^-(v_k), ent_{3(k+1)} = d^+(v_{k+1}) = (k+1) - d^-(v_k)$ and, \\
$ent_{1(k+2)} = k +2, ent_{2(k+2)} = d^-(v_{k+2})= d^-(v_k)+1, ent_{3(k+1)} = d^+(v_{k+2}) = (k+2) - d^-(v_k) - 1$ and, \\
$ent_{1(k+3)} = k +3, ent_{2(k+3)} = d^-(v_{k+3})= d^-(v_k)+1, ent_{3(k+3)} = d^+(v_{k+3}) = (k+3) - d^-(v_k) - 1.$  \\ \\
\noindent Let $n = 2k - d^-(v_k)$ and it easily follows from Step 5 that $\Bbb J(J_n(1)) = \{v_k\}.$ Now let $n = 2k - d^-(v_k) + 1$ and initialise $\Bbb J(J_{n+1}(1)) = \{v_k\}$ and set $t = k.$ Also let $l = k+1, k+2, ..., 2k - d^-(v_k).$ \\ \\
For $l= k+1$ we have that $(2k - d^-(v_k) +1) - (k+1) + d^-(v_k) = k = t,$ so $v_{k+1} \in \Bbb J(J_{n+1}(1)).$ \\
For $l=k+2$ we have that $(2k - d^-(v_k) +1) - (k+2) + d^-(v_k) +1 = k = t,$ so $v_{k+2} \in \Bbb J(J_{n+1}(1)).$ \\
For $l=k+3$ we have that $(2k - d^-(v_k) +1) - (k+3) + d^-(v_k) +1 = k-1 \neq t,$ so $v_{k+3} \notin \Bbb J(J_{n+1}(1)).$ \\ \\
So it follows that if $d^-(v_k)=d^-(v_{k+1})$ and $d^-(v_{k+2})=d^-(v_{k+1})+1$ then $\Bbb J(J_n(1)) =\{v_k\}$ and $\Bbb J(J_{n+1}(1)) = \{v_k, v_{k+1},v_{k+2}\}$ with $n \in \{2k - d^-(v_k), 2k - d^-(v_k) +1\}.$ \\ \\
Conversely, if $\Bbb J(J_n(1)) =\{v_k\}$ and $\Bbb{J}(J_{n+1}(1))= \{v_k,v_{k+1},v_{k+2}\}$ we have from the inverse of definition 1.7 and Steps 1, 2, 3, 4, 5 and 6 of the Fisher Algorithm the associated entries: 

\noindent $ent_{1n} = n, ent_{2n} = (n-k), ent_{3n} = k, ent_{4n} = \{v_k\}, ent_{5n} = k \Rightarrow$ \\
$ent_{1k} = k, ent_{2k} = 2k-n, ent_{3k} = n-k, ent_{5k} = k \Rightarrow$ \\
$ent_{1(k+1)} = k+1, ent_{2(k+1)} = 2k-n, ent_{3(k+1)} = n- k + 1, ent_{5(k+1)} = k \Rightarrow$ \\
$ent_{1(k+2)} = k+2, ent_{2(k+2)} = 2k-n+1, ent_{3(k+2)} = n- k + 1, ent_{5(k+2)} = k+1.$ \\ 
$\therefore d^-(v_k)=d^-(v_{k+1})$ and $d^-(v_{k+2})=d^-(v_{k+1})+1.$ \\ \\
Result $(b)$ follows similarly to $(a)$.
\end{proof}
\begin{proposition} 
For all Jaco Graphs $J_n(1),$ we have Card $\Bbb{J}(J_n(1))\leq3.$ 
\end{proposition}
\begin{proof}
It is evident that for some $m\in\Bbb{N},$ Card $\Bbb{J}(J_m(1))=3.$ Let $\Bbb{J}(J_m(1))=\{v_k,v_{k+1},v_{k+2}\}.$ So in Step 4 of the Fisher Algorithm we initially set $i=m$ and $t=k.$ We also have that $i - (k+2)+ d^-(v_{k+2})= t.$\\ \\
From Morrie's theorem it follows that $d^-(v_k)=d^-(v_{k+1})$ and $d^-(v_{k+2})=d^-(v_{k+1})+1.$ It follows that, $d^-(v_{k+3})=d^-(v_{k+1})+1.$ However, in Step 5 we have $i-(k+3)+ d^-(v_{k+3})= i-(k+3)+d^-(v_{k+1})+1= (i-(k+2)+d^-(v_{k+2}))-1<t.$ So vertex $v_{k+3}$ cannot be added to $\Bbb J(J_m(1)).$
\end{proof}
\begin{corollary}
From Proposition 2.4 it follows that if  and only if the Jaconian vertex of $J_{i-1}(1),~i\geq2$ is unique say, $v_k$ then $\Bbb{J}(J_i(1)) =$ either $\{v_k,v_{k+1}\}$ or $\{v_k,v_{k+1},v_{k+2}\}.$
\end{corollary}
\begin{proof}
By extending from to $J_{i-1}(1)$ to $J_i(1)$ the edge $(v_k,v_i)$ is not linked. Because $d(v_{k+1})= d(v_k)-1$ in $J_{i-1}(1)$ and increases by $1$ in $J_i(1)$ it follows that $d(v_{k+1})=d(v_k)$ in $J_i(1).$ Hence, at least $\Bbb J_i(1)=\{v_k,v_{k+1}\}.$ If  $d^-(v_{k+2})= d^-(v_{k+1})+1,$ then $\Bbb J_i(1)=\{v_k, v_{k+1},v_{k+2}\}.$  So it follows that $\Bbb J_i(1)=$ either $\{v_k,v_{k+1}\}$ or $\{v_k,~v_{k+1},~v_{k+2}\}.$\\  \\
Conversely, assume that $\Bbb J_i(1)=$ either $\{v_k,~v_{k+1}\}$ or $\{v_k,~v_{k+1},~v_{k+2}\}.$\\ 
Case 1:  Let $\Bbb {J}(J_i(1))=\{v_k,~v_{k+1}\}.$  So $i - (k+1)+ d^-(v_{k+1})= k$ in $J_i(1).$ Hence in $J_{i-1}(1)$ we have $(i-1)-(k+1)+ d^-(v_{k+1})= (i- (k+1))+ d^-(v_{k+1}) - 1= k-1.$ So from Step 5 of the Fisher Algorithm it follows that $v_{k+1}\notin\Bbb{J}(J_{i-1}(1)).$ However, $v_k\in\Bbb{J}(J_{i-1}(1))=\{v_k\}.$ \\ \\
Case 2: Let $\Bbb{J} (J_i(1))=\{v_k,~v_{k+1},~v_{k+2}\}.$ Same reasoning as in case 1, follows.
\end{proof}
\begin{corollary}
If $k+ d^+(v_k)= i$ and $(k+1)+ d^+(v_{k+1})> i+1$ then $v_k$ is the unique Jaconian vertex of $J_i(1).$
\end{corollary}
\begin{proof}
The result follows directly from Step 5 of the Fisher Algorithm.
\end{proof}
\begin{proposition}
If we have $d^-(v_{k-1})= d^-(v_k)= d^-(v_{k+1})$ then $v_k$  is the unique Jaconian vertex of $J_l(1),~l=2k - d^-(v_k).$
\end{proposition}
\begin{proof}
For $d^+(v_k)=k - d^-(v_k)$ and $l=k+ d^+(v_k)= k+ k - d^-(v_k)=2k - d^-(v_k)$ it follows that $v_k$  is a Jaconian vertex of $J_l(1).$
Furthermore, $v_k$ is the unique Jaconian vertex of $J_l(1)$ because:\\
Case 1: For $v_{k-1}$ and because $d^-(v_{k-1})= d^-(v_k),$  we have $l-(k-1)+d^-(v_k)=2k-d^-(v_k)-(k-1)+ d^-(v_k)=2k-d^-(v_k)-k+1+d^-(v_k)=2k-k+1=k+1>k.$ Hence, $v_{k-1}$ is not a Jaconian vertex of $J_l(1).$\\ \\
Case 2: For $v_{k+1}$ and because $d^-v_{k+1}) = d^-(v_k),$ we have $l-(k+1)+d^-(v_k)=2k-d^-(v_k)-(k+1)+d^-(v_k)=2k-d^-(v_k)-k-1+d^-(v_k)=2k-k-1=k-1<k.$ Hence, $v_{k+1}$ is not a Jaconian vertex of $J_l(1).$
\end{proof}
\begin{proposition}
$\Bbb{J}(J_{k-1}(1))=\{v_{l-1}\}$ if and only if $d^+(v_k)=d^+(v_{k+1})=l.$
\end{proposition} 
\begin{proof}
If $\Bbb {J}(J_{k-1}(1)=\{v_{l-1}\},$ implying $\Delta(J_{k-1}(1)) = l - 1,$ it follows from Step 2 of the Fisher Algorithm that $d^-(v_k) = (k-1) - \Delta (J_{k-1}) = (k-1) - d^+(v_{k-1}) .$ So because $k = (k-1) + 1,$ it follows that $d^+(v_k) = l$ because $d(v_i) = d^+(v_i) + d^-(v_i),\forall i.$ By similar reasoning $d^+(v_{k+1}) = l,$ so it holds that $d^+(v_k)=d^+(v_{k+1})=l.$ \\ \\
Conversely, if $d^+(v_k) = d^+(v_{k+1}) = l,$ it follows by inversing the convergence properties of the Fisher Algorithm that $\Bbb{J}(J_{k-1}(1))=\{v_{l-1}\}.$
\end{proof}
\begin{theorem}
Let $m=n+ \Delta(J_n(1)),$ then $\Delta(J_m(1))=$ either $n$ or $n-1.$
\end{theorem}
\begin{proof}
Let $m=n+ \Delta(J_n(1)).$\\
Case 1: Assume  $\Bbb{J}(J_n(1))=\{v_k\}$ so $\Delta(J_n(1))= k.$  We have that $d^+(v_n)=k$ so in $J_m(1)$ the edge $(v_n, v_m)$ is defined to attain $d(v_n)=n,$ and $v_n$ is the prime Jaconian vertex of $J_m(1),~m=n+ \Delta(J_n(1)).$ Hence, $\Delta(J_m(1))=n.$\\ \\
Case 2: Assume $\Bbb{J}(J_n(1))=\{v_k,v_{k+1}\}$ or $\{v_k,v_{k+1},v_{k+2}\}.$ We have that $d^+(v_n)=k$ or $k+1.$ So by the same reasoning as in Case 1 it follows that $\Delta(J_m(1))=n$ or $n-1.$
\end{proof}               
\begin{theorem}
[Conjectured] For the Jaco Graphs $J_n(1),$ $J_m(1)$ with $n\geq 3,~m\geq 3, n \neq m$ we have 
\begin{equation*}
\Delta(J_{n+m}(1)) =
\begin{cases}
\Delta(J_n(1))+ \Delta(J_m(1)), & \text{if $J_n(1)$ or $J_m(1)$ has a unique Jaconian vertex}\\
\Delta(J_n(1))+ \Delta(J_m(1)) + 1, & \text{otherwise.} 
\end{cases}
\end{equation*}
\end{theorem}
\begin{theorem}
If the Jaco Graph $J_n(1)$ has a unique Jaconian vertex (prime Jaconian vertex only) at $v_i,$ then:\\ \\
(a) Edge $(v_i, v_n)$ exists and,\\
(b) $\Delta(J_n(1))+d(v_n)= n.$
\end{theorem}
\begin{proof}
The proof follows through contra absurdum.
Assume the Jaco Graph $J_n(1)$ has a unique Jaconian vertex $v_i$. If the edge $(v_i, v_n)$ is undefined then at most, the edge $(v_i, v_{n-1})$ is defined. When considering vertex $v_{i+1}$ and proceeding with construction per definition, at least the vertex $v_{i+1},$ can at most, be linked to $v_n$ to have the edge $(v_{i+1}, v_n)$ defined.  
So, $d(v_{i+1})\geq d(v_i)= \Delta(J_n(1)),$ renders a contradiction on the uniqueness of the Jaconian vertex $v_i.$ Through contra absurdum we conclude that $(v_i, v_n)$ is defined. Hence, result $(a)$ follows.\\ \\
The Hope subgraph on the vertices $v_{i+1},v_{i+2},\cdots,v_n$ allows for $d(v_n)=(n-i)-1.$ But with the edge $(v_i, v_n)$ added we have $d(v_n)=(n-i)-1+1=n-i.$\\
$\therefore$ $\Delta(J_n(1))+d(v_n)=i+(n-i)= n.$ Hence, result $(b)$ follows.
\end{proof}    
\noindent Note that $\Delta(J_n(1))+ d(v_n) = n \nRightarrow$ uniqueness of the Jaconian vertex.
\begin{theorem}
Consider the Jaco Graph $J_n(1).$ For $m<i<k\leq n,$ the edge $(v_m, v_i)$ can only exist if the edge $(v_m, v_{i-1})$ exists. Furthermore, if the edge $(v_i, v_k)$ exists then the edges $(v_{i+1}, v_k),\cdots, (v_{k-1}, v_k)$ exist.
\end{theorem}
\begin{proof}
(Part 1): After applying definition 1.1 exhaustively to the vertex $v_{m-1}$ the vertex $v_m$ has attained $d^-(v_m).$ Applying definition 1.1 exhaustively to vertex $v_m$ proceeds by linking the edges $(v_m, v_{m+1}), (v_i, v_{m+2}),\cdots$ in such a way as to attain $d(v_m)=\max(abs(\min(d(v_m)))\leq m.$ So after linking the edge $(v_m, v_{i-1})$ and, if and only if $d^*(v_m)@v_{i-1}<m,$ can the edge $(v_m, v_i)$ be linked.\\ \\
(Part 2): If the edge $(v_i, v_k)$ exists then $d^*(v_i)@v_k\leq i.$ So because $i+1>i$ it follows that $d^*(v_{i+1})@v_k<i+1.$ Hence, by definition 1.1 the edge $(v_{i+1}, v_k)$ exists.
\end{proof}
\begin{lemma}
The vertex $v_i$ is the prime Jaconian vertex of a Jaco Graph $J_n(1),$ if and only $d(v_{l})\leq d(v_i)=i$ for $l=i+1,i+2,\cdots,n.$
\end{lemma}
\begin{proof}
Let the vertex $v_i$ be the prime Jaconian vertex of the Jaco Graph $J_n(1),~n\in\Bbb{N}.$ If for any $v_{l}, l=i+1,i+2,\cdots,n.$ we have $d(v_{l})>d(v_i)$ then $v_i$ cannot be the prime Jaconian vertex of $J_n(1)$ as it then, contradicts definition 1.3. \\ \\
Conversely: If $d(v_{l}) \leq d(v_i)=i$ for $l=i+1,i+2,\cdots,n,$ then if follows from definition 1.2 and 1.3 as well as from property 1, $(d(v_i)>d(v_m)=m$ for all $m\in\{1,2,3,\cdots,i -1\}),$ that vertex $v_i$ is the prime Jaconian vertex.
\end{proof}
\begin{theorem}
If for the Jaco Graph $J_n(1),$ we have $\Delta(J_n(1))= k,$ then the out-degrees of the vertices $v_{k+1}, v_{k+2}, v_{k+3},\cdots, v_n$ are respectively, $\lceil d^+(v_{k+1}) \rceil= (n-k-1),\lceil d^+(v_{k+2}) \rceil= (n-k-2),\cdots, \lceil d^+(v_{n-1}) \rceil =1, \lceil d^+(v_n) \rceil=0.$
\end{theorem}
\begin{proof}
From definition 1.4 we have that with $v_k$ the prime Jaconian vertex, the Hope subgraph $\Bbb{H}(J_n(1)),$ on vertices $v_{k+1}, v_{k+2}, v_{j+3},\cdots,v_n$ is a complete graph. On applying definition 1.1 exhaustively to vertex $v_{k+1}$ it is clearly possible to link $n-(k+1)$ edges hence, $\lceil d^+(v_{k+1}) \rceil= n-(k+1)= n-k-1.$\\ \\
Furthermore, it follows from definition 1.4 that the subgraph on vertices $v_{k+2}, v_{k+3}, v_{j+4},\cdots,v_n$ is a complete graph as well. On applying definition 1.1 exhaustively to vertex $v_{k+2}$ it is clearly possible to link $n- (k + 2)$ edges hence, $\lceil d^+(v_{k+2}) \rceil = n-(k + 2)=n-k-2.$\\ \\
By repeating the immediate above to vertices $v_{k+3},\cdots, v_n$  and noting that $\lceil d^+(v_n) \rceil = n- n = 0,$ the result follows.
\end{proof}
\begin{theorem}
If for the Jaco Graph $J_n(1),$ we can express $n=7+3k,$ $k\in\{0,1,2,\cdots\},$ then:
$\Delta(J_n(1))\leq n-(3+k)$ and, edge $(v_{\Delta(J_n(1))}, v_n)$ exists.
\end{theorem}
\begin{proof}
For $k=0$ it follows from the Fisher Algorithm that $\Delta (J_7(1))\leq7–(3+0).$ Furthermore, since $4 + d^+(v_4)=7$ the edge $(v_4, v_7)$ exists.\\ \\
Assume the result holds for $m=7+3k,~k>0,~k\in\{1, 2,\cdots\}$ so for the Jaco Graph $J_m(1)$ we have $\Delta(J_m(1))\leq m-(3 + k)$ and, edge $(v_{\Delta(J)}, v_m)$ exists.\\ \\
Now consider the Jaco Graph on $n$ vertices with, $n = 7 + 3(k + 1).$ Noting that $\Delta(J_\ell(1)) \leq \ell-1, \forall \ell \in \Bbb N,$ it follows that $\Delta(J_n(1))- \Delta(J_m(1))<7+3(k+1)-(7+3k)=3$ hence,  $\Delta(J_n(1))<3+ \Delta(J_m(1))<3+m-(3+k)=n-(3 +k).$ So it follows that $\Delta(J_n(1))\leq n-(3+k)-1=n-(3+(k + 1)).$
\end{proof}
\section{Fibonaccian-Zeckendorf Result}
\begin{theorem}
(Bettina's Theorem): Let $\Bbb{F} = \{f_0, f_1,f_2,f_3, ...\}$ be the set of Fibonacci numbers and let $n=f_{i_1} + f_{i_2} + ... + f_{i_r}, n\in \Bbb N$ be the Zeckendorf representation of $n.$ Then 
\begin{center}
$d^+(v_n) = f_{i_1-1} + f_{i_2-1} + ... +f_{i_r-1}.$
\end{center}
\end{theorem}
\begin{proof}
Through induction we have that first of all, $1=f_2$ and $d^+(v_1) =1=f_1.$ Let $2 \leq n = f_{i_1} + f_{i_2} + ... + f_{i_r}$ and let $k = f_{i_1-1} + f_{i_2-1} + ... + f_{i_r-1}.$ If $i_r \geq 3,$ then $k = f_{i_1-1} + f_{i_2-1} + ... + f_{i_r-1}$ is the Zeckendorf representation of $k$, such that induction yields $d^+(v_k) = k = f_{i_1-2} + f_{i_2-2} + ... + f_{i_r-2}.$ Since $k + d^+(v_k)  = f_{i_1-1} + f_{i_1-2} + f_{i_2-1} + f{i_2-2} + ... f_{i_r-1} + f_{i_r-2} = f_{i_1} + f_{i_2} + ... f_{i_r} = n,$ we have $d^+(v_n) = k.$ \\ \\
Finally consider $n=f_{i_1} + f_{i_2} + ... + f_{i_r}, i_r =2.$ Note that $n > 1$ implies that $i_{r-1} \geq 4$ and that the Zeckendorf representation of $n-1$ given by $n-1 = f_{i_1} + f_{i_2} + ... + f_{i_{r-1}}.$ Let $k= d^+(v_{n-1}).$ Through induction we have that, $k = f_{i_1-1} + f_{i_2-1} + ... + f_{i_{r-1}-1},$ and since $i_{r-1} \leq 4,$ this is the Zeckendorf representation of $k.$ Accordingly, $d(v_k) = f_{i_1-2} + f_{i_2-2}+ ... + f_{i_{r-1}-2},$ and $k+ d^+(v_k) =  f_{i_1-1} + f_{i_1-2} + f_{i_2-1} + f{i_2-2} + ... f_{i_{r-1}-1} + f_{i_{r-1}-2} = n-1.$ It follows that $d^+(v_n) > k = d^+(v_{n-1}).$ Hence, it follows that $d^+(v_n) = k+1 = (f_{i_1-1} + f_{i_2-1} + ... + f_{i_{r-1}-1}) + f_1 = f_{i_1-1} + f_{i_2-1} + ... + f_{i_r-1}.$
\end{proof}
\begin{proposition} 
For Fibonacci numbers $t= f_{n-1},~h=f_n$ and $l=f_m,~t\geq3,~h\geq3,~l\geq3$ we have:\\ \\
\noindent (a) $\Delta(J_h(1))=t=f_{n-1},$ \\
(b) $\Bbb{J}(J_h(1))=\{v_t\},$ \\
(c) $\Delta(J_{h+l}(1))=\Delta(J_h(1)) + \Delta(J_l(1)),$ \\
(d) $\Bbb J(J_{h+ l}(1)) = \{v_{\Delta(J_h(1)) + \Delta(J_l(1))}\}.$
\end{proposition}
\begin{proof}
From the Fisher Algorithm it follows that for $s=f_2,~u=f_3,w=f_4$ we have:\\
$f_2=1,~d^+(v_s)=1,~\Delta(J_s(1))=0$ and $\Bbb{J}(J_1(1)) = \{v_1\};$\\
$f_3=2,~d^+(v_u)=1,~\Delta(J_u(1))=1$ and $\Bbb{J}(J_2(1))=\{v_1, v_2\};$\\
$f_4=3,~d^+(v_w)=2,~\Delta(J_w(1))=2$ and $\Bbb{J}(J_3(1))=\{v_2\}.$\\ \\
So assume for $i=f_{k-1}$ and $r=f_k$ we have $d^+(v_r)=f_{k-1}, \Delta(J_r(1))=f_{k-1}$ and $\Bbb{J}(J_r(1))=\{v_i\}.$ So for $q=f_{k+1},$ we have that:\\
\begin{center}
$f_{k+1}=f_k+f_{k-1}=f_k+d^+(v_r)\Rightarrow v_r$
\end{center}
is the prime Jaconian vertex of $J_q(1)$ so $\Delta(J_q(1))=r=f_k.$ So result $(a)$ namely, $\Delta(J_h(1))=t=f_{n-1}$ holds in general.\\ \\
Bettina’s Theorem yields $d^+(v_{r+1})=d^+(v_r)+1.$ So we have $(f_k+1)+d^+(v_{r+1})=(f_k+1)+d^+(v_r)+1=(f_k+d^+(v_r))+2>f_{k+1}+1.$ From Corollary 2.6 it follows that $\Bbb{J}(J_q(1))=\{v_r\}.$ So result $(b)$ namely, $\Bbb{J}(J_h(1))=\{v_t\}$  holds in general.\\ \\
The result $\Delta(J_{h+l}(1))=\Delta(J_h(1)) + \Delta(J_l(1))$ follows from $(b)$ read together with Theorem 2.10. \\ \\
The result $\Bbb J(J_{h+ l}(1)) = \{v_{\Delta(J_h(1)) + \Delta(J_l(1))}\}$ follows from $(c)$ read together with Theorem 2.10.
\end{proof}
\noindent [Open problem: Note that for $P_0 =\{n\in \Bbb N|1 \leq n \leq 44\}$ we have, $\Delta(J_n(1)) = 3k$ if $n = 5k.$ Then for $P_1 = \{n \in \Bbb N|45 \leq n \leq 99\}$ we have, $\Delta(J_n(1)) =3k+1$ if $n=5k.$ Then for $P_2 = \{n \in \Bbb N|100\leq n \leq ????\}$ we have, $\Delta(J_n(1)) = 3k + 2$ if $n = 5k$ and seemingly so on $\cdots.$ Find a partitioning $\Bbb P = \cup_{i\in \Bbb N}P_i$ to close the result.]\\ 
\noindent [Open problem: Prove Lemma (conjectured) 2.2.]\\
\noindent [Open problem: Prove Theorem (conjectured) 2.10.] \\
\noindent [Open problem: Determine the number of spanning trees of $J_n(1)$.] \\ \\
\textbf{\emph {Open access:}} This paper is distributed under the terms of the Creative Commons Attribution License which permits any use, distribution and reproduction in any medium, provided the original author(s) and the source are credited.   

\noindent References \\ \\
$[1]$  Bondy, J.A., Murty, U.S.R., \emph {Graph Theory with Applications,} Macmillan Press, London, 1976. \\ 
$[2]$ Zeckendorf, E., \emph {Repr\'esentation des nombres naturels par une somme de nombres de Fibonacci ou de nombres de Lucas,} Bulletin de la Soci\'et\'e Royale des Sciences de Li\`ege, Vol 41, (1972): pp 179-182. \\  \\
\noindent Acknowledgement will be given to colleagues for preliminary peer review and other contributions on the content of this paper during the preprint arXiv publication term:\footnote{} \\ \\
\noindent {\scriptsize [Remark: The concept of Jaco Graphs followed from a dream during the night of 10/11 January 2013 which was the first dream Kokkie could recall about his daddy after his daddy passed away in the peaceful morning hours of 24 May 2012, shortly before the awakening of Bob Dylan, celebrating Dylan's 71st birthday]}\\ 

\end{document}